\documentclass[12pt]{amsart}
\usepackage{amsmath}
\usepackage[lmargin=1in, rmargin=1in, bmargin=1.1in]{geometry}
\usepackage{amsthm}
\usepackage{amsfonts}
\usepackage{comment}
\usepackage{amssymb}
\usepackage{eufrak}
\usepackage{comment}
\usepackage{hyperref}
\usepackage{mathtools}
\usepackage{multicol}

\numberwithin{equation}{section}
\usepackage{breqn}

\newtheorem{Theorem}{Theorem}[section]
\newtheorem{Corollary}[Theorem]{Corollary}
\newtheorem{Lemma}[Theorem]{Lemma}

\theoremstyle{definition}

\theoremstyle{remark}
\newtheorem*{remark}{Remark}

\newcommand{\lcm}{\mathrm{lcm}}

\usepackage{cite}

\title[McKay-Thompson coefficients and distributions of the moonshine module]{Coefficients of McKay-Thompson series and distributions of the moonshine module}

\author[Hannah Larson]{Hannah Larson}
\address{Department of Mathematics, Harvard University, Cambridge, MA 02138}
\email{hannahlarson@college.harvard.edu}

\begin{document}

\maketitle

\begin{abstract}
In a recent paper, Duncan, Griffin and Ono provide exact formulas for the coefficients of McKay-Thompson series and use them to find asymptotic expressions for the distribution of irreducible representations in the moonshine module $V^\natural = \bigoplus_n V_n^\natural$. 
Their results show that as $n$ tends to infinity, $V_n^\natural$ is dominated by direct sums of copies of the regular representation. That is, if we view $V_n^\natural$ as a module over the group ring $\mathbb{Z}[\mathbb{M}]$, the free-part dominates. A natural problem, posed at the end of the aforementioned paper, is to characterize the distribution of irreducible representations in the non-free part. Here, we study asymptotic formulas for the coefficients of McKay-Thompson series to answer this question. We arrive at an ordering of the series by the magnitude of their coefficients, which corresponds to various contributions to the distribution. In particular, we show how the asymptotic distribution of the non-free part is dictated by the column for conjugacy class 2A in the monster's character table. 
We find analogous results for the other monster modules $V^{(-m)}$ and $W^\natural$ studied by Duncan, Griffin, and Ono.
\end{abstract}

\section{Introduction}
\textit{Monstrous moonshine} refers to the mysterious connection between the representation theory of the largest sporadic simple group $\mathbb{M}$, known as the \textit{monster group}, and the theory functions which are invariant under the action of certain subgroups of $\mathrm{GL}_2(\mathbb{Q})^+$, known as \textit{modular functions}.
The first hints of monstrous moonshine came in the famous observations of McKay and Thompson,
\begin{equation} \label{eq1}
1 = 1, \qquad 196884 = 196883+1, \qquad 21493760=1+196883+21296876,
\end{equation}
among others.
The numbers on the left are coefficients in the Fourier expansion of the \textit{normalized $j$-function}, the modular function defined by
\begin{align}
J(\tau) &:= j(\tau) - 744 = \frac{E_4(\tau)}{\Delta(\tau)} - 744 \\
&= \sum_{n=-1}^\infty c(n)q^n = q^{-1} + 196884q + 21493760q^2 + \ldots, \notag
\end{align}
where $E_4(\tau)$ is the Eisenstein series of weight $4$, $\Delta(\tau)$ is the modular discriminant, and $q:= e^{2\pi i \tau}$.
 Meanwhile, the numbers in the summands on the right-hand side are dimensions of irreducible representations of the monster group.

The striking coincidences in \eqref{eq1}, along with many similar observations, led Thompson to conjecture \cite{T} the existence of a naturally-defined infinite-dimensional monster module, $V^\natural = \bigoplus_{n=-1}^\infty V_n^\natural$, satisfying 
\[\dim(V_n^\natural) = c(n).\]
Thompson then considered the other graded trace-functions that would arise from $V^\natural$,
\begin{equation}
T_g(\tau) :=q^{-1} + \sum_{n=1}^\infty c_g(n) := \sum_{n=-1}^\infty \text{tr}(g|V_n^\natural)q^n,
\end{equation}
for $g \in \mathbb{M}$, known as \textit{McKay-Thompson series}, and found that their coefficients were also expressible as simple sums of entries in the monster's character table.

The distinguishing feature of $J(\tau)$ is that it generates the field of $\mathrm{SL}_2(\mathbb{Z})$-invariant functions on the upper-half plane with at most exponential growth as $\mathfrak{I}(\tau) \rightarrow \infty$, and is the unique such function with Fourier expansion beginning $q^{-1} + O(q)$. Given any subgroup $\Gamma \subset \mathrm{GL}_2(\mathbb{Q})^+$ commensurable with $\mathrm{SL}_2(\mathbb{Z})$, if the field of $\Gamma$-invariant functions with at most exponential growth at the cusps is generated by a single function, then we
 call the unique generator with Fourier expansion beginning $q^{-1} + O(q)$ the \textit{Hauptmodul} for $\Gamma$. The famous monstrous moonshine conjecture of Conway and Norton states that for each $g \in \mathbb{M}$, the McKay-Thompson series $T_g(\tau)$ is the Hauptmodul for a particular discrete subgroup $\Gamma_g \subset\mathrm{GL}_2(\mathbb{Q})^+$, which they describe explicitly \cite{CN}. In highly celebrated works, Frenkel, Lepowsky and Murman constructed the moonshine module $V^\natural$ in \cite{FLM, FLM2}, and Borcherds proved the monstrous moonshine conjecture in \cite{B}.
 
More recent work shows that the monster also plays a role in quantum gravity, in particular in theories of chiral three-dimensional gravity \cite{DGO,85}. These theories take interest in a tower of monster modules
 \[V^{(-m)} = \bigoplus_{n=-m}^\infty V_n^{(-m)}\]
 with $V^{(-1)} = V^\natural$. In \cite{DGO}, the authors consider the \textit{order $m$ McKay-Thompson series} arising from these modules,
\begin{equation}
T_g^{(-m)}(\tau) :=q^{-m} + \sum_{n=1}^\infty c_g(-m, n) q^n := q^{-m} + \sum_{n=1}^\infty \text{tr}(g|V_n^{(-m)})q^n, \end{equation}
for $g \in \mathbb{M}$, and show that they coincide with particular Rademacher sums of order $m$ for the group $\Gamma_g$ (see Theorem 7.1 of \cite{DGO}).
 
 It turns out that the multiplicities of irreducible representations in $V_n^{(-m)}$
 are expressible as sums of entries in the monster's character table weighted by the McKay-Thompson coefficients $c_g(-m,n)$. 
 We first give exact formulas for the coefficients $c_g(-m,n)$, similar to those in Theorem 8.12 of \cite{DGO}, which are better suited to our purposes.
 
 \begin{Theorem}\label{Th1}
  Let $g \in \mathbb{M}$, and let $\mathcal{W}_g$ be the set corresponding to the Atkin-Lehner involutions of $\Gamma_g$ given in the appendix. For any positive integers $m$ and $n$, the coefficients of $T_g^{(-m)}(\tau)$ are exactly
 \[c_g(-m,n) = \sum_{e \in \mathcal{W}_g} 2\pi \sqrt{\frac{em}{n}} \sum_{c>0} \frac{K_c(g,e, -m, n)}{c}I_1\left(\frac{4\pi}{c}\sqrt{emn} \right) ,\]
 where $K_c(g,e,-m,n)$ is the Kloosterman sum defined in \eqref{defK} and $I_1(x)$ is the Bessel function of the first kind.
 \end{Theorem}

From well-known asymptotics of the Bessel function and explicit computations of Kloosterman sums, we arrive at the following asymptotic formula for these coefficients.
 \begin{Theorem} \label{core} Let $g \in \mathbb{M}$ and let $\Gamma_g = N||h+\mathcal{W}_g$ as listed in the appendix. If $\varepsilon = \mathrm{max} (\mathcal{W}_g)$, then as $n \rightarrow \infty$, we have
 \[c_g(-m, n) \sim \frac{(m\varepsilon)^{1/4}}{\sqrt{2N}n^{3/4}}\cdot K_N(g,\varepsilon,-m,n)\cdot \mathrm{exp}\left(\frac{4\pi \sqrt{\varepsilon mn}}{N}\right),\]
with the following exceptions:
\begin{enumerate}
\item If $g$ is in conjugacy class $\mathrm{8D}$ or $\mathrm{8E}$, $m$ is odd, and $n \equiv -m \pmod 4$, then
\[c_g(-m, n) \sim \frac{m^{1/4}}{4\sqrt{2}n^{3/4}} \cdot K_{16}(g, 1, -m, n) \cdot \mathrm{exp}\left(\frac{4\pi \sqrt{mn}}{16}\right).\]
\item If $g$ is in conjugacy class $\mathrm{16B}$, $m \equiv 2 \pmod 4$, and $n \equiv -m \pmod 8$, then
\[c_g(-m,n) \sim \frac{m^{1/4}}{8n^{3/4}} \cdot K_{32}(g,1,-m,n) \cdot \mathrm{exp}\left(\frac{4\pi\sqrt{mn}}{32}\right)\]
\item If $g$ is in conjugacy class $\mathrm{18A}$, $m \not\equiv 0 \pmod 3$, and $n \equiv -m \pmod 3$, then
\[c_g(-m,n) \sim \frac{m^{1/4}}{6n^{3/4}}\cdot K_{18}(g,1,-m,n)\cdot \mathrm{exp}\left(\frac{4\pi \sqrt{mn}}{18}\right).\]
\item If $g$ is in conjugacy class $\mathrm{24D}$, $m$ is odd, and $n \equiv m \pmod 4$, then
\[c_g(-m, n) \sim \frac{m^{1/4}}{4\sqrt{3}n^{3/4}}\cdot K_{24}(g,1,-m,n) \cdot \mathrm{exp}\left(\frac{4\pi\sqrt{mn}}{24}\right).\]
\end{enumerate}
 Hence, the list of conjugacy classes in the appendix beginning 
 \[\mathrm{1A, 2A, 3A, \{2B, 4A\}, 5A, \{6A, 6B\}, 7A, (4B, 8A), (3C, 3B, 9A), \ldots}\]
orders the order $m$ McKay-Thompson series by the asymptotic magnitude of their coefficients, where we remove a conjugacy class if $K_N(g, \varepsilon, -m, n)$, which is periodic in $m$ and $n$ modulo $N/\varepsilon$, vanishes.
 \end{Theorem}
 
 
The ordering of conjugacy classes in Theorem \ref{core} can be interpreted as an order for contributions of the corresponding columns in the character table to the multiplicities of irreducible representations in $V_n^{(-m)}$ as $n \rightarrow \infty$. Let $M_i$, for $1 \leq i \leq 194$, denote the 194 inequivalent irreducible representations of the monster group, labeled so that the character of $M_i$ is the function denoted $\chi_i$ in \cite{57}. In addition, let $\mathbf{m}_i(-m,n)$ be the multiplicity of $M_i$ in $V_n^{(-m)}$ so that
\[V_n^{(-m)} \simeq \bigoplus_{i=1}^{194} M_i^{\oplus \mathbf{m}_i(-m,n)}.\]
In \cite{DGO}, the authors prove that as $n \rightarrow \infty$,
 \begin{align} \label{b}
 \mathbf{m}_i(-m,n) &\sim \frac{\dim(\chi_i)m^{1/4}}{\sqrt{2} n^{3/4} |\mathbb{M}|} \cdot e^{4\pi \sqrt{mn}} .
 \end{align}
In particular, the limit
 \begin{align*}
 \delta(\mathbf{m}_i(-m)) := \lim_{n \rightarrow \infty} \frac{\mathbf{m}_i(-m,n)}{\sum_{i=1}^{194} \mathbf{m}_i(-m,n)} 
  \end{align*}
 exists and is given by
 \begin{equation} \label{DGOcor}
 \delta(\mathbf{m}_i(-m)) =\frac{\dim(\chi_i)}{\sum_{j=1}^{194}\dim(\chi_j)} = \frac{\dim(\chi_i)}{5844076785304502808013602136}.
 \end{equation}
 
Another way to phrase this result is that, as $n$ tends to infinity, $V_n^{(-m)}$ tends to direct sums of copies of the regular representation. 
Viewing $V^{(-m)}$ as a module over the group ring $\mathbb{Z}[\mathbb{M}]$, it is thus natural to decompose $V_n^{(-m)}$ into a free part and a non-free part and ask what can be said about the distribution of irreducible representations in the non-free part (see Problem 10.9 in \cite{DGO}).
Define $\mathbf{nf}_i(-m,n)$ to be the multiplicity of $M_i$ in the non-free part of $V_n^{(-m)}$  so that
 \[V_n^{(-m)} \simeq \mathbb{Z}[\mathbb{M}]^{\oplus \mathbf{f}(-m,n)} \oplus \bigoplus_{i=1}^{194}M_i^{\oplus \mathbf{nf}_i(-m,n)}\]
where $\mathbf{f}(-m,n)$ is maximal.
\begin{remark}
The results of \cite{DGO} show that as $n \rightarrow \infty$, we have
\[\mathbf{f}(-m,n) \sim \frac{m^{1/4}}{\sqrt{2}n^{3/4}|\mathbb{M}|}\cdot e^{4\pi \sqrt{mn}}
.\]
\end{remark}

Here, we use the order in Theorem \ref{core} to find that
the column of the monster's character table corresponding to conjugacy class 2A dictates the asymptotic non-free distributions of $V_n^{(-m)}$.
Writing $\chi_i(\mathrm{2A})$ for the value of the character of $M_i$ on an element in conjugacy class 2A and $|\mathrm{2A}|$ for the size of this conjugacy class, our precise result can be stated as follows.
 \begin{Theorem} \label{nf}
 For any $1 \leq i \leq 194$, as $n \rightarrow \infty$, we have
  \[ \mathbf{nf}_i(-m,n) \sim 
\left( \frac{\dim(\chi_i)}{3301375}+\chi_i(\mathrm{2A})\right)\cdot \frac{|\mathrm{2A}| m^{1/4}}{|\mathbb{M}|(2n)^{3/4}}\cdot e^{2\pi\sqrt{2mn}}.\]
\end{Theorem}
Thus the following limit is well-defined
\[\delta(\mathbf{nf}_i(-m)) := \lim_{n \rightarrow \infty} \frac{\mathbf{nf}_i(-m,n)}{\sum_{i=1}^{194}\mathbf{nf}_i(-m,n)}.
\]
\begin{Corollary} \label{1.4}
In particular, as $n \rightarrow \infty$, we have that
\[\delta(\mathbf{nf}_i(-m)) = \frac{\frac{\dim(\chi_i)}{3301375}+\chi_i(\mathrm{2A})}{\sum_{j=1}^{194}\frac{\dim(\chi_j)}{3301375}+\chi_j(\mathrm{2A})} = \frac{\dim(\chi_i)+3301375\chi_i(\mathrm{2A})}{5845054856224474627181019136}.\]
 \end{Corollary}
 
 \begin{remark}
 It is natural to compare Theorem \ref{nf} and Corollary \ref{1.4} to \eqref{b} and \eqref{DGOcor}. 
The appearance of two character values in our expression for the non-free distribution comes from accounting for negative character values, which are not considered in \cite{DGO}.
 \end{remark}
 
 \begin{remark}
These methods can also be applied to the monster module $W^\natural = \bigoplus_n W_n^\natural$ studied by Duncan, Griffin, and Ono in their answer to a problem of Witten on the distribution of black holes (see Question 6.1 and \S 8 of \cite{DGO}). 
If we let $\mathbf{\widetilde{n}f}_i(n)$ be the multiplicity of $M_i$ in the non-free part of $W_n^\natural$, then as $n \rightarrow \infty$, we have
\[\mathbf{\widetilde{n}f}_i(n) \sim  \left(\frac{\dim(\chi_i)}{330175} +\chi_i(\mathrm{2A})\right)\frac{|\mathrm{2A}| \sqrt{12}}{|\mathbb{M}|\sqrt{24n+1}}\cdot e^{\frac{\pi}{6}\sqrt{11(24n+1)}},\]
and so
\[\delta(\mathbf{\widetilde{n}f}_i) = \frac{\dim(\chi_i)+3301375\chi_i(\mathrm{2A})}{5845054856224474627181019136},\]
providing a refinement to these distribution results.
 \end{remark}
 
 The following table, which should be compared with Table 1 in \cite{DGO}, illustrates the asymptotics of Corollary \ref{1.4} for $\chi_1, \chi_{2}, \chi_{17}$ and $\chi_{194}$ when $m=1$. Let $\delta(\mathbf{nf}_i(-1,n))$ be the proportion of irreducible representations corresponding to $\chi_i$ in the non-free part of $V_n^{(-1)} = V_n^\natural$.
 \begin{center}
\begin{tabular}{|c|c|c|c|c|} 
\hline
\rule{0pt}{2.5ex} $n$ & $\delta(\mathbf{nf}_1(-1, n))$ &$ \delta(\mathbf{nf}_{2}(-1, n))$ & $ \delta(\mathbf{nf}_{17}(-1, n))$&$ \delta(\mathbf{nf}_{194}(-1, n))$  \\[0.3ex]
\hline
\rule{0pt}{2.2ex}-1 & 1 & 0 & 0 & 0\\
0 & 0 & 0 & 0 & 0\\
1 & 1/2 & 1/2 & 0 & 0\\
2 & 1/3 & 1/3 & 0 &0 \\
$\vdots$ & $\vdots$ & $\vdots$ & $\vdots$ & $\vdots$ \\ 
40 & $4.011 \ldots \times 10^{-4}$ & $2.514\ldots \times 10^{-3}$ & 0 & $0.00891\ldots$ \\
80 & $4.809\ldots \times 10^{-14}$ & $7.537\ldots \times 10^{-13}$ & $1.411\ldots \times 10^{-12}$& $0.04428\ldots$ \\
120 & $1.329\ldots \times 10^{-17}$ & $5.304\ldots \times 10^{-16}$ & $4.256\ldots \times 10^{-13}$ & $0.04428\ldots$ \\
160 & $1.642\ldots \times 10^{-20}$ & $4.972\ldots \times 10^{-18}$ & $3.949\ldots \times 10^{-19}$ &
$0.04428\ldots$ \\
200 & $7.663\ldots \times 10^{-22}$ & $2.572\ldots \times 10^{-18}$ & $6.748\ldots \times 10^{-25}$ &
$0.04428\ldots$ \\
240 & $5.787\ldots \times 10^{-22}$ & $2.478\ldots \times 10^{-18}$ & $1.914\ldots \times 10^{-30}$ & 
$0.04428\ldots$ \\
280 & $5.664\ldots \times 10^{-22}$ & $2.470\ldots \times 10^{-18}$ & $8.932\ldots \times 10^{-36}$ & 
$0.04428\ldots$ \\
320 & $5.650\ldots \times 10^{-22}$ & $2.469\ldots \times 10^{-18}$ & $2.154\ldots \times 10^{-39}$ &
$0.04428\ldots$ \\
360 & $5.648\ldots \times 10^{-22}$ & $2.468 \ldots \times 10^{-18}$ & $7.602\ldots \times 10^{-44}$ &
$0.04428\ldots$  \\
400 & $5.648 \ldots \times 10^{-22}$ & $2.468 \ldots \times 10^{-18}$ & $3.626\ldots \times 10^{-47}$ &
$0.04428\ldots$ \\
$\vdots$ & $\vdots$ & $\vdots$ & $\vdots$ & $\vdots$ \\
$\infty$ & $\frac{403}{713507672878964187888308}$ & $\frac{3523073}{1427015345757928375776616}$ & 0 & $\frac{63189970139661766171875}{1427015345757928375776616}$  \\[1.5ex]
\hline
\end{tabular}
 \end{center}

 \vspace{.15in}
The exact values in the bottom row have the following decimal approximations:
\begin{align*}
\delta(\mathbf{nf}_1(-1)) &= \frac{403}{713507672878964187888308} \approx 5.648\ldots \times 10^{-22} \\
\delta(\mathbf{nf}_2(-1)) &= \frac{3523073}{1427015345757928375776616} \approx 2.468 \times 10^{-18} \\
\delta(\mathbf{nf}_{17}(-1)) &= 0 \\
\delta(\mathbf{nf}_{194}(-1)) &= \frac{63189970139661766171875}{1427015345757928375776616} \approx 0.04428\ldots.
\end{align*}

This paper is organized as follows. In Section \ref{s2}, we recall the description of the moonshine groups and construct Poincar\'e series for these groups which are equal to the McKay-Thompson series. We then find exact expressions for their coefficients following the methods in \cite{BO, Brunier}. Next, in Section \ref{s3}, we study the Kloosterman sums appearing in these exact expressions to arrive at the asymptotics in Theorem \ref{core}. In Section \ref{snf}, we apply the asymptotic formulas to describe the non-free distributions.

\subsection*{Acknowledgements}
This research was started at the 2015 REU at Emory University. The author would like to thank Ken Ono for advising this project, Michael Griffin and Lea Beneish for helpful conversations and suggestions, and the NSF for its support.

\section{McKay-Thompson coefficients}\label{s2}

Recall that $\mathrm{GL}_2(\mathbb{Q})^+$ acts on the upper half-plane $\mathbb{H}$ by fractional linear transformation: for $\gamma = \left(\begin{smallmatrix} a & b \\ c & d \end{smallmatrix}\right) \in \mathrm{GL}_2(\mathbb{Q})^+$ and $\tau \in \mathbb{H}$, we set $\gamma \tau := \frac{a\tau + b}{c\tau + d}$. Let $\Gamma \subset \mathrm{GL}_2(\mathbb{Q})^+$ be a subgroup commensurable with $\mathrm{SL}_2(\mathbb{Z})$.
 A \textit{modular function} for $\Gamma$ is a meromorphic function $f: \mathbb{H} \rightarrow \mathbb{C}$ which is invariant under the action of $\Gamma$, i.e. satisfies $f(\gamma \tau) = f(\tau)$ for all $\gamma \in \Gamma$.

\subsection{The moonshine groups}
We now describe the groups $\Gamma_g$ having Hauptmodules which are the McKay-Thopson series $T_g(\tau)$.
Recall that for a positive integer $N$, the congruence subgroup $\Gamma_0(N) \subseteq \mathrm{SL}_2(\mathbb{Z})$ is defined as those matrices which are upper triangular mod $N$,
\[\Gamma_0(N) := \left\{\left(\begin{matrix} a&b\\c&d\end{matrix}\right)  \in \mathrm{SL}_2(\mathbb{Z}) : c \equiv 0 \pmod N  \right\}.\]
For each exact divisor $e \mid N$ with $(e, N/e) = 1$, the associated \textit{Atkin-Lenher involutions} for $\Gamma_0(N)$ are those matrices of the form
\[W_e = \left(\begin{matrix} ae & b \\ cN & de \end{matrix}\right)\]
with determinant $e$. For each $g \in \mathbb{M}$, we associate a group $E_g$, denoted by a symbol $\Gamma_0(N|h)+e,f, \ldots$ (or just $N|h+e,f\ldots$), where $h$ divides $(N, 24)$ and each $e, f,\ldots$ exactly divides $N/h$. 
This symbol stands for the group of matrices
\[\Gamma_0(N|h)+e,f, \ldots := \left(\begin{matrix} 1/h & 0 \\ 0 & 1 \end{matrix}\right) \langle \Gamma_0(N/h), W_e, W_f, \ldots \rangle \left(\begin{matrix} h & 0 \\ 0 & 1 \end{matrix}\right),\]
where $W_e, W_f,\ldots$ are representative of Atkin-Lehner involutions for $\Gamma_0(N/h)$. We write $\mathcal{W}_g := \{1, e, f, \ldots\}$ for the set corresponding to Atkin-Lehner involutions in $E_g$.
The groups $E_g$ are eigengroups for $T_g(\tau)$, meaning that $T_g(\gamma \tau) = \sigma_g(\gamma) T_g(\tau)$ for $\gamma \in E_g$, where $\sigma_g$ is a group homomorphism from $E_g$ to the $h$th roots of unity. Conway and Norton provide the following explicit description of $\sigma_g$ on generators of $E_g$.

\begin{Lemma}[Conway-Norton] \label{sigma}
With the above notation, we have
\begin{enumerate}
\item $\sigma_g(\gamma) = 1$ if $\gamma \in \Gamma_0(Nh)$ 
\item $\sigma_g(\gamma) = 1$ if $\gamma$ is an Atkin-Lehner involution of $\Gamma_0(Nh)$ inside $E_g$
\item $\sigma_g(\gamma) = e^{\frac{-2\pi i}{h}}$ if $\gamma = \left(\begin{matrix} 1 & 1/h \\ 0 & 1 \end{matrix}\right)$ \label{haha}
\item $\sigma_g(\gamma) = e^{-\lambda_g \frac{2\pi i }{h}}$ if $\gamma =\left(\begin{matrix} 1 & 0 \\ N & 1 \end{matrix}\right)$,
\end{enumerate}
where $\lambda_g$ is $-1$ if $N/h \in \mathcal{W}_g$ and $1$ otherwise.
\end{Lemma}

The group $\Gamma_g$ is the kernel of $\sigma_g$ inside $E_g$, which is denoted by $\Gamma_0(N||h)+e, f, \ldots$ (or just $N||h+e,f,\ldots$). A complete list of the groups $\Gamma_g$ is in the appendix.

\subsection{Poincar\'e series and exact expressions}
In Section 8.3 of \cite{DGO}, the authors build a nice basis of Maass-Poincar\'e series for $\Gamma_0(N)$ and take certain combinations of them to obtain the McKay-Thompson series. Here, we carry out the alternative construction suggested at the beginning of that section, constructing Poincar\'e series for the groups $\Gamma_g$ directly.

Keeping with the notation of \cite{BO,Brunier,DGO}, let
\[\mathcal{M}_s(w) := M_{0,s-\tfrac{1}{2}}(|w|),\]
where $M_{\nu,\mu}(z)$ is the $M$-Whittaker function. In addition, let $\tau = x+iy$ with $x, y \in \mathbb{R}, y> 0$, and let
\[\phi_s(\tau) := \mathcal{M}_s(4\pi y) e^{2\pi i x}.\]
Given a subgroup $\Gamma \subset \mathrm{GL}_2(\mathbb{Q})^+$ commensurable with $\mathrm{SL}_2(\mathbb{Z})$, let $\Gamma_\infty$ denote the stabilizer of $\infty$ in $\Gamma$. For each positive integer $m$, the Poincar\'e series
\[\mathcal{P}_s(m, \Gamma) := \sum_{M \in \Gamma_\infty\backslash \Gamma} \phi_s(-m \cdot M\tau)\]
converges for $s$ with $\mathrm{Re}(s) > 1$.  In general, these Poincar\'e series are harmonic Maass forms, but when $\Gamma$ has genus zero they are weakly holomorphic. In this case, taking the limit as $s \rightarrow 1+$ along the real axis, we obtain a modular function $\mathcal{P}(m, \Gamma)$ for $\Gamma$ having a Fourier expansion $q^{-m} + O(q)$ at infinity and no other poles.
Since there is a unique such function, we must have
\[T_g^{(-m)}(\tau) = \mathcal{P}(m, \Gamma_g).\]
Thus, we can use the method in \cite{BO, Brunier} for determining coefficients of these Poincar\'e series to prove Theorem \ref{Th1}. For a positive integer $c$, any $g \in \mathbb{M}$, and integers $m$ and $n$, we define the \textit{Kloosterman sum} 
\begin{equation} \label{defK}
K_c(g,e,m,n) := \sum_{\left(\begin{smallmatrix}a&b\\c&d\end{smallmatrix}\right) \in \mathcal{F}_{g,e}(c)} \text{exp}\left(\frac{2\pi i(ma+nd)}{c}\right),
\end{equation}
where the sum ranges over matrices in
\[\mathcal{F}_{g,e}(c) := \left\{\left(\begin{matrix}a&b\\c&d \end{matrix}\right) \in \Gamma_g : 0 \leq a, d < c, \ ad-bc = e\right\}.\]
We can now prove our exact expression for the McKay-Thompson coefficients.

\begin{proof}[Proof of Theorem \ref{Th1}]
We argue exactly as in Section 1.3 of \cite{Brunier}, but using the group $\Gamma_g$ and a slight modification to account for matrices with different determinants. Since the imaginary part of $\tau$ transforms under $\left(\begin{smallmatrix} a & b \\ c & d \end{smallmatrix}\right)$ by $y \mapsto \frac{(ad-bc)y}{|c\tau+d|^2},$ we 
need to include an extra factor of $e$ in front of $y$ for each Atkin-Lehner involution $W_e$ in the sum over $\Gamma_\infty \backslash \Gamma_g/ \Gamma_\infty$ in the last equation on page 31 of \cite{Brunier}. Following through, this results in introducing a sum over $e \in \mathcal{W}_g$, using the modified Kloosterman sum in \eqref{defK}, and changing the quantity $B$ on page 33 to $\frac{|m|}{c^2 ey}$. Plugging this into the equation in terms of $A$ and $B$ on page 33 and carrying out the simplification leading to Proposition 1.10, we obtain the expression in our Theorem \ref{Th1}.
\end{proof}

\section{Asymptotic formulas for coefficients} \label{s3}
First recall the asymptotics of the $I$-Bessel function,
\begin{equation}
I_1(x) \sim \frac{e^x}{\sqrt{2\pi x}}\left(1 - \frac{3}{8x}+\ldots \right).
\end{equation}
Suppose $\Gamma_g = N||h+\mathcal{W}_g$ for some $g \in \mathbb{M}$. The Kloosterman sum $K_c(g, e, -m,n)$ vanishes for all $c \not\equiv 0 \pmod N$ since $\mathcal{F}_{g,e}(c)$ is empty in this case.
Therefore, when it does not vanish, the term dominating the expression for $c_g(-m,n)$ in Theorem \ref{Th1} is
\[2\pi \sqrt{\frac{m\varepsilon}{n}} \cdot  \frac{K_N(g,\varepsilon, -m, n)}{N} \cdot I_1\left(\frac{4\pi}{N}\sqrt{\varepsilon mn}\right) \sim \frac{(m\varepsilon)^{1/4}}{\sqrt{2N}n^{3/4}} \cdot K_N(g,\varepsilon, -m, n) \cdot  \mathrm{exp}\left(\frac{4\pi \sqrt{\varepsilon mn}}{N}\right) , \]
where $\varepsilon = \max (\mathcal{W}_g)$.
Thus, to prove Theorem \ref{core}, it suffices to show that if 
$K_N(g,\varepsilon,-m,n)$ vanishes, then $c_g(-m, n)$ also vanishes or we are in one of the listed exceptions.

The functions $K_N(g,\varepsilon,-m,n)$ are not hard to calculate: using the procedure on page 35 of \cite{DGO} one may evaluate $\sigma_g(\gamma)$ for all $\gamma = \left(\begin{smallmatrix} a&b\\ c&d \end{smallmatrix}\right) \in E_g$ with $0 \leq a, d<c$ and $\det(\gamma) = e$ to find $\mathcal{F}_{g,e}(c)$ and then sum over this finite set.

\begin{remark}
We note that there is a typo on page 35 of \cite{DGO}. The last $b$ in the second bullet point should be an $a$.
\end{remark}

For convenience, we provide a collection of simplified expressions for the Kloosterman sums which allows us to compute them all explicitly and determine when they vanish. Let $\zeta_\ell := e^{2\pi i /\ell}$.

\begin{Lemma} \label{2.2}
Suppose $\Gamma_g = N||h+\mathcal{W}_g$. The following are true:
\begin{enumerate}
\item If $N/h \in \mathcal{W}_g$, then 
\[K_N(g, N/h, -m, n) = \begin{cases} h &\text{if $n \equiv -m \pmod h$} \\ 0 & \text{otherwise.}\end{cases}\]
\item If $h=1$, and $\ell = N/e$ for some $e \in \mathcal{W}_g$, then
\[K_N(g,e,-m,n) = \sum_{a \in (\mathbb{Z}/\ell\mathbb{Z})^\times}\zeta_\ell^{-am+(ae)^{-1}n}.\]
In particular, $K_N(g, \varepsilon, -m, n) \neq 0$ if and only if one of the following holds:
\begin{itemize}
\item $n \equiv -m \pmod \ell$ for $g$ and $\ell$ in the table below
\vspace{.1in}
\begin{center}
\begin{tabular}{|c||c|c|c|c|c|c|c|c|c|}
\hline
\rule{0pt}{2.5ex}g & $\mathrm{4C} $ & $\mathrm{9B}$ & $\mathrm{12E}$ & $\mathrm{12I}$ & $\mathrm{18D}$ & $\mathrm{28C}$ &$\mathrm{36B}$\\
\hline
\rule{0pt}{2.2ex} $\ell$ & $\mathrm{2}$ & $\mathrm{3}$ & $\mathrm{2}$ & $\mathrm{2}$ &$\mathrm{3}$ & $\mathrm{2}$ & $\mathrm{3}$\\[0.2ex]
\hline
\end{tabular}
\end{center}
\vspace{.1in}
\item $g$ is in $\mathrm{8E}$ and $n \equiv m \pmod 4$
\item $g$ is in $\mathrm{16B}$ and $n \equiv -m \pmod 4$ if $m$ is odd or $n \equiv -m \pmod 8$ if $m$ is even
\item $g$ is in $\mathrm{18A}$ and $n \equiv m \pmod 3$.
\end{itemize}
\item If $N/2h \in \mathcal{W}_g$ then
\[K_N(g, N/2h, -m, n) = \begin{cases} (-1)^{(m+n)/h} \cdot h&\text{if $n \equiv -m \pmod h$} \\ 0 & \text{otherwise.} \end{cases}\]
\end{enumerate}
The leading Kloosterman sums that do not fit into one of the above cases are
\begin{align*}
\mathrm{8D} &: & K_8(g, 1, -m, n) &= \zeta_8^{3m+n}+\zeta_8^{m+3n}+\zeta_8^{7m+5n}+\zeta_8^{5m+7n}, \\
\intertext{which is non-zero if and only if $n \equiv m \!\!\! \pmod 4$.}
\mathrm{12G} &: & K_{12}(g, 2, -m, n) &= \zeta_6^{m+n}+\zeta_6^{5m+5n}+\zeta_6^{2m+2n}+\zeta_6^{4m+4n}, \\
\intertext{which is non-zero if and only if $n \equiv -m \!\!\! \pmod 2$.}
\mathrm{15D} &: & K_{15}(g, 1, -m, n) &= \zeta_{15}^{4m+n} + \zeta_{15}^{2m+2n} + \zeta_{15}^{3m+3n} + \zeta_{15}^{m+4n} + \zeta_{15}^{9m+6n}+ \zeta_{15}^{7m+7n} \\
&&&\quad +\zeta_{15}^{6m+9n} + \zeta_{15}^{14m+11n}
 + \zeta_{15}^{12m+12n}+ \zeta_{15}^{13m+13n} + \zeta_{15}^{11m+14n}, \\
\intertext{which is non-zero if and only if $n \equiv -m \!\!\! \pmod 3$.}
\mathrm{24D} &: & K_{24}(g, 3, -m, n) &= \zeta_8^{m+n} + \zeta_8^{3m+3n} + \zeta_8^{5m+5n} + \zeta_8^{7m+7n}, \\
\intertext{which is non-zero if and only if $n \equiv -m \!\!\! \pmod 4$.}
\mathrm{24G} &: & K_{24}(g, 2, -m, n) &= \zeta_{12}^{m+n} + \zeta_{12}^{2m+2n} + \zeta_{12}^{4m+4n} + \zeta_{12}^{5m+5n} + \zeta_{12}^{7m+7n}\\
&&&\quad  + \zeta_{12}^{8m+8n} + \zeta_{12}^{10m+10n} + \zeta_{12}^{11m+11n},
\end{align*}
which is non-zero if and only if $n \equiv -m \pmod 4$.
\end{Lemma}

In order to prove Theorem \ref{core}, we need to know that the McKay-Thompson coefficients vanish on certain arithmetic progressions corresponding to the vanishing of the Kloosterman sums. To this end, 
we define a sieving operator which acts on functions $f: \mathbb{H} \rightarrow \mathbb{C}$ with Fourier expansions $f(\tau) = \sum_{n} a(n)q^n$ by
\begin{equation*}
(f |S_{\ell,r})(\tau):= \sum_{n \equiv r \! \! \! \pmod \ell} a(n)q^n =
\frac{1}{\ell}\sum_{k=0}^{\ell-1}\zeta_\ell^{\ell-kr}f\left(\tau + \frac{k}{\ell}\right).
\end{equation*}
In addition, we define the \textit{projection operator} as the sum of this operator over a square class
\[(f| \widetilde{S}_{\ell, r})(\tau) := \sum_{s \in \mathcal{S}_r} (f|S_{\ell, s})(\tau),\]
where $\mathcal{S}_r :=\{r a^2 \pmod \ell : (a, \ell) = 1 \}$.
It is well-known that sieving on Fourier coefficients in an arithmetic progression preserves modularity, although it may change the group on which the function is modular. 
More precisely, if $f(\tau)$ is modular on $\Gamma_0(N)$, then $(f|\widetilde{S}_{\ell,r})(\tau)$ is modular on $\Gamma_0(\lcm(N, \ell^2))$.

We now prove several lemmas that determine which arithmetic progressions certain McKay-Thompson series are supported on.

\begin{Lemma} \label{ding}
For $g$ in one of the conjugacy classes listed below, $T_g^{(-m)}(\tau)$ has coefficients $c_g(-m,n)$ supported on $n \equiv -m \pmod \ell$ where $\ell$ is specified in the table.
\end{Lemma}
\begin{center}
\begin{tabular}{|c||c|c|c|c|c|c|c|c|c|c|}
\hline
\rule{0pt}{2.5ex}$g$ & $\mathrm{4C} $ & $\mathrm{8E}$ & $\mathrm{9B}$ & $\mathrm{12E}$ & $\mathrm{12I}$ & $\mathrm{16B}$ & $\mathrm{18D}$ & $\mathrm{28C}$ &$\mathrm{36B}$\\
\hline
\rule{0pt}{2.2ex} $\ell$ & 2 & 2 & 3 & 2 & 2 & 4 &3 & 2 & 3\\[0.2ex]
\hline
\end{tabular}
\end{center}

\begin{proof}
Note that for these $\ell$ we have $\widetilde{S}_{\ell, r} = S_{\ell, r}$. We observe that for the given $\ell$, the function $(T^{-m}_g|S_{\ell, -m})(\tau)$ is again modular on $\Gamma_g$. Then we need only check that the principal part of the Fourier expansion at each cusp is unaffected. The principal part at infinity is already supported on this arithmetic progression, and the projection operator cannot introduce poles at the other cusps if there are no others to begin with. Hence, $(T^{-m}_g|S_{\ell, -m})(\tau) = T^{-m}_g(\tau)$ so the coefficients are supported on the claimed arithmetic progression.
\end{proof}

We need additional results for $g$ in conjugacy class 8D, 8E, 16B, 18A, or 24D when $m$ is even.
 For ease of notation, let $T_{\mathrm{8D}}^{(-m)}(\tau), T_{\mathrm{8E}}^{(-m)}(\tau)$, etc.~denote $T_g^{(-m)}(\tau)$ for $g$ in the specified conjugacy class.
 
 \begin{Lemma} \label{ids}
 The following identities relating order $m$ McKay-Thompson series are true:
 \begin{align*}
 T_{\mathrm{8D}}^{(-2)}(\tau) = T_{\mathrm{8E}}^{(-2)}(\tau) &= T_{\mathrm{4C}}(2\tau)
& T_{\mathrm{16B}}^{(-2)}(\tau) &= T_{\mathrm{8E}}(2\tau) \\
 T_{\mathrm{8D}}^{(-4)}(\tau) = T_{\mathrm{8E}}^{(-4)}(\tau) &= T_{\mathrm{2B}}(4\tau) 
 & T_{\mathrm{16B}}^{(-4)}(\tau) &= T_{\mathrm{4C}}(4\tau)\\
  T_{\mathrm{24D}}^{(-2)}(\tau) &= T_{\mathrm{12E}}(2\tau)
 & T_{\mathrm{16B}}^{(-8)}(\tau) &= T_{\mathrm{2B}}(8\tau) \\
  T_{\mathrm{24D}}^{(-4)}(\tau) &= T_{\mathrm{6C}}(4\tau)
 & T_{\mathrm{18A}}^{(-3)}(\tau) &= T_{\mathrm{6D}}(3\tau)
 \end{align*}
 \end{Lemma}
 
 \begin{proof}
 In each case, we note that both sides are modular on $\Gamma_g$ for $g$ in the conjugacy class in the subscript on the left. In addition, both have Fourier expansion $q^{-m} + O(q)$ at infinity and no other poles. Since there is a unique such function, they must be equal.
 \end{proof}

\begin{remark}
The identities involving 8E and 16B are special cases of Lemma 2.11 of \cite{me} which provides an expression for $T_g^{(-m)}(\tau)$ as a combination of Hauptmoduln on lower levels hit with Hecke operators for all $g$ with $\Gamma_g = \Gamma_0(N)$. The other identities above suggest that this result generalizes to other moonshine groups.
\end{remark}

\begin{Lemma} \label{more}
Suppose $m$ is even and $g$ is in conjugacy class $\mathrm{8D, 8E, 16B, 18A}$ or $\mathrm{24D}$. Then $c_g(-m,n) = 0$ for all $n \not\equiv \pm m \pmod \ell$ for $\ell = 4, 4, 8, 3$ or $4$ respectively. When $g$ is in $\mathrm{18A}$, this holds for all $m$.
\end{Lemma}

\begin{proof}
When $m = \ell$, the claim follows from Lemma \ref{ids}. Suppose that the claim holds when $m = i\ell$ for all $0 < i \leq k$. We can write $T_g^{(-(k+1)\ell)}(\tau)$ as $T_g^{(-\ell)}(\tau) \cdot T_g^{(-k\ell)}(\tau)$ plus a linear combination of $T_g^{(-i\ell)}(\tau)$ for $i \leq k$. Each term has coefficients supported on the arithmetic progression $ n\equiv 0 \pmod \ell$, so the claim holds for all $m \equiv 0 \pmod \ell$ by induction. For $g$ not in 18A, if $m = \ell/2$, Lemma \ref{ids} together with Lemma \ref{ding} shows that $T_g^{(-m)}(\tau)$ has coefficients $c_g(-m,n)$ supported on the progression $n \equiv \ell/2 \pmod \ell$. As before, suppose the claim holds when $m = \ell/2 + i\ell$ for all $0 \leq i \leq k$. We can write $T_g^{(-\ell/2 -(k+1)\ell)}(\tau)$ as $T_g^{(-\ell/2)}(\tau) \cdot T_g^{(-(k+1)\ell)}(\tau)$ plus a linear combination of $T_g^{(-\ell/2-i\ell)}(\tau)$ for $i \leq k$. The claim now follows by induction and the $m \equiv 0 \pmod \ell$ case.

If $g$ is in 16B, we need the two base cases $m = 2$ and $m= 6$. These follow from Lemma \ref{ids} and Lemma \ref{ding} and the fact that $T_{\mathrm{16B}}^{(-6)}(\tau)$ is a linear combination of $T_{\mathrm{16B}}^{(-4)}(\tau)T_{\mathrm{16B}}^{(-2)}(\tau)$ and $T_{\mathrm{16B}}^{(-2)}(\tau)$. Finally, for $g$ in 18A, we use the identity
\[T_{\mathrm{18A}}(\tau) = T_{\mathrm{18D}}(\tau) - \frac{2}{T_{\mathrm{18D}}(\tau)},\]
together with Lemma \ref{ding}, to establish the $m = 1$ and $m = 2$ base cases. Arguing by induction as before completes the proof.
\end{proof}

\begin{Lemma} \label{dong}
If $\Gamma_g = N||h + \mathcal{W}_g$ then $c_g(-m,n) = 0$ for all $n \not\equiv -m \pmod h$.
\end{Lemma}

\begin{proof}
If we let $\gamma = \left(\begin{smallmatrix} 1 & 1/h \\ 0 & 1 \end{smallmatrix}\right)$, then by Lemma \ref{sigma} (\ref{haha}), we have 
\[T_g(\tau) = \frac{1}{h}\sum_{j=0}^{h-1} e^{2\pi i j/h}T_g(\gamma^j \tau) = (T_g|S_{h,-1})(\tau),\]
showing that $T_g(\tau)$ is supported on terms $q^n$ with $n \equiv -1 \pmod h$. Since $T_g^{(-m)}(\tau)$ is a polynomial in $T_g(\tau)$, only terms $T_g(\tau)^k$ with $k \equiv m \pmod h$ may appear. Hence, $T_g^{(-m)}(\tau)$ is supported on terms $q^n$ with $n \equiv -k \equiv -m \pmod h$.
\end{proof}

We can now prove Theorem \ref{core}.

\begin{proof}[Proof of Theorem \ref{core}]
Using Lemmas \ref{ding}, \ref{more}, and \ref{dong} and the explicit computations of Kloosterman sums, we see that whenever $K_N(g, \varepsilon, -m, n)$ vanishes, so does $c_g(-m, n)$, except for the special cases listed in the theorem. In these cases, it is clear that the listed term dominates and it is not hard to check that the Kloosterman sum does not vanish. The ordering listed in the appendix follows from comparing the arguments of the exponentials.
\end{proof}

\section{Non-free distributions} \label{snf}
At the beginning of Section 8.6 of \cite{DGO}, the authors use the orthogonality of characters to deduce that the functions defined by
\[T_{\chi_i}^{(-m)}(\tau) := \frac{1}{|\mathbb{M}|} \sum_{g \in \mathbb{M}}\chi_i(g) T_g^{(-m)}(\tau)\]
are generating functions for the multiplicities $\mathbf{m}_i(-m,n)$. In particular,
\begin{equation}\label{ft}
\mathbf{m}_i(-m, n) = \frac{1}{|\mathbb{M}|}\sum_{g \in \mathbb{M}} \chi_i(g) c_g(-m,n).
\end{equation}
By Theorem \ref{core}, we have that $c_1(-m,n)$ exponentially dominates $c_g(-m,n)$ for all other $g \in \mathbb{M}$. Thus we can see, as concluded in \cite{DGO}, that as $n \rightarrow \infty$,
\[\mathbf{m}_i(-m,n) \sim \frac{c_{1}(-m,n)}{|\mathbb{M}|}\dim(\chi_i).\]
The first order correction to this phenomenon comes from the terms weighted by the next largest McKay-Thompson coefficients in \eqref{ft}. With this idea, we can prove Theorem \ref{nf}.

\begin{proof}[Proof of Theorem \ref{nf}]
 By Theorem \ref{core}, the coefficients $c_g(-m,n)$ for $g$ in conjugacy class 2A dominate all others with $g$ not in 1 or 2A. The Kloosterman sum $K_2(g,2,-m,n) = 1$ for all $m$ and $n$ so we can write
 \[c_g(-m,n) \sim \frac{m^{1/4}}{(2n)^{3/4}} \cdot e^{2\pi\sqrt{2mn}}.\]
The number of regular representations we can pull out of $V_n^\natural$ is limited by the character with the smallest ratio  of $\chi_i(\mathrm{2A})$ to $\dim(\chi_i)$. This minimum is achieved with the characters we have labeled $\chi_{16}$ and $\chi_{17}$, which have
\[\frac{\chi_{16}(\mathrm{2A})}{\dim(\chi_{16})} = \frac{\chi_{17}(\mathrm{2A})}{\dim(\chi_{17})} = -\frac{1}{3301375}.
 \]
Therefore, when we add up the contributions to the non-free part from \eqref{ft} for $g$ in conjugacy classes 1 and 2A, we find that as $n \rightarrow \infty$,
 \begin{align*}
 \mathbf{nf}_i(-m,n)  &\sim \frac{|\mathrm{2A}|}{|\mathbb{M}|} \left(\frac{c_g(-m,n)}{3301375}\dim(\chi_i) +  c_g(-m,n)\chi_i(\mathrm{2A})\right) \\
 &\sim \left(\frac{\dim(\chi_i)}{3301375}+\chi_i(\mathrm{2A})\right) \frac{|\mathrm{2A}|m^{1/4}}{|\mathbb{M}|(2n)^{3/4}}\cdot e^{2\pi\sqrt{2mn}},
 \end{align*}
 as desired.
\end{proof}

\section{Appendix}

\subsection*{Moonshine groups} Here we list the symbols $\Gamma_g = N||h + e, f, \ldots$ for each conjugacy class of the monster. As in \cite{CN}, if $h = 1$, we omit the $``||1"$, if $\mathcal{W}_g = \{1\}$ then we just write $N||h$, and if $\mathcal{W}_g$ contains every exact divisor of $N/h$ then we write $N||h+$. The labeling of conjugacy classes is as in \cite{57}.

\begin{multicols}{3}
\begin{tabular}{ll}
1A  & $ 1 $\\
2A  & $ 2+ $\\
2B  & $ 2 $\\
3A  & $ 3+ $\\
3B  & $ 3 $\\
3C  & $ 3||3 $\\
4A  & $ 4+ $\\
4B  & $ 4||2+ $\\
4C  & $ 4 $ \\
4D  & $ 4||2 $\\
5A  & $ 5+ $\\
5B  & $ 5 $\\
\end{tabular}

\columnbreak

\begin{tabular}{ll}
6A  & $ 6+ $\\
6B  & $ 6+6 $\\
6C  & $ 6+3 $\\
6D  & $ 6+2 $\\
6E  & $ 6 $\\
6F  & $ 6||3 $\\
7A  & $ 7+ $\\
7B  & $ 7 $\\
8A  & $ 8+ $\\
8B  & $ 8||2+ $ \\
8C  & $ 8||4+ $\\
8D  & $ 8||2 $\\
\end{tabular}

\columnbreak

\begin{tabular}{ll}
8E  & $ 8 $\\
8F  & $ 8||4 $\\
9A  & $ 9+ $\\
9B  & $ 9 $\\
10A  & $ 10+ $\\
10B  & $ 10+5 $\\
10C  & $ 10+2 $\\
10D  & $ 10+10 $\\
10E  & $ 10 $\\
11A  & $ 11+ $\\
12A  & $ 12+ $\\
12B  & $ 12+4 $ \\

\end{tabular}

\columnbreak

\begin{tabular}{ll}
12C  & $ 12||2+ $\\
12D  & $ 12||3+ $\\
12E  & $ 12+3 $\\
12F  & $ 12||2+6 $\\
12G  & $ 12||2+2 $\\
12H  & $ 12+12 $\\
12I  & $ 12 $\\
12J  & $ 12||6 $\\
13A  & $ 13+ $\\
13B  & $ 13 $\\
14A  & $ 14+ $\\
14B  & $ 14+7 $\\
14C  & $ 14+14 $\\
15A  & $ 15+ $\\
15B  & $ 15+5 $\\
15C  & $ 15+15 $\\
15D  & $ 15||3 $\\
16A  & $ 16||2+ $\\
16B  & $ 16 $\\
16C  & $ 16+ $\\
17A  & $ 17+ $\\
18A  & $ 18+2 $\\
18B  & $ 18+ $\\
18C  & $ 18+9 $\\
18D  & $ 18 $\\
18E  & $ 18+18 $\\
19A  & $ 19+ $\\
20A  & $ 20+ $\\
20B  & $ 20||2+ $\\
20C  & $ 20+4 $\\
20D  & $ 20||2+5 $\\
20E  & $ 20||2+10 $\\
20F  & $ 20+20 $\\
21A  & $ 21+ $\\
21B  & $ 21+3 $\\
21C  & $ 21||3+ $\\
21D  & $ 21+21 $\\
22A  & $ 22+ $\\
22B  & $ 22+11 $\\
23AB  & $ 23+ $\\
24A  & $ 24||2+ $\\
24B  & $ 24+ $\\
24C  & $ 24+8 $\\
24D  & $ 24||2+3 $\\
24E  & $ 24||6+ $\\
24F  & $ 24||4+6 $\\
\end{tabular}

\columnbreak

\begin{tabular}{ll}
24G  & $ 24||4+2 $\\
24H  & $ 24||2+12 $\\
24I  & $ 24+24 $\\
24J  & $ 24||12 $\\
25A  & $ 25+ $\\
26A  & $ 26+ $\\
26B  & $ 26+26 $\\
27A  & $ 27+ $\\
27B  & $ 27+ $\\
28A  & $ 28||2+ $\\
28B  & $ 28+ $\\
28C  & $ 28+7 $\\
28D  & $ 28||2+14 $\\
29A  & $ 29+ $\\
30A  & $ 30+6,10,15 $\\
30B  & $ 30+ $\\
30C  & $ 30+3,5,15 $\\
30D  & $ 30+5,6,30 $\\
30E  & $ 30||3+10 $\\
30F  & $ 30+2,15,30 $\\
30G  & $ 30+15 $\\
31AB  & $ 31+ $\\
32A  & $ 32+ $\\
32B  & $ 32||2+ $\\
33A  & $ 33+11 $\\
33B  & $ 33+ $\\
34A  & $ 34+ $\\
35A  & $ 35+ $\\
35B  & $ 35+35 $\\
36A  & $ 36+ $\\
36B  & $ 36+4 $\\
36C  & $ 36||2+ $\\
36D  & $ 36+36 $\\
38A  & $ 38+ $\\
39A  & $ 39+ $\\
39B  & $ 39||3+ $\\
39CD  & $ 39+39 $\\
40A  & $ 40||4+ $\\
40B  & $ 40||2+ $\\
40CD  & $ 40||2+20 $\\
41A  & $ 41+ $\\
42A  & $ 42+ $\\
42B  & $ 42+6,14,21 $\\
42C  & $ 42||3+7 $\\
42D  & $ 42+3,14,42 $\\

\end{tabular}

\columnbreak

\begin{tabular}{ll}
44AB  & $ 44+ $\\
45A  & $ 45+ $\\
46AB  & $ 46+23 $\\
46CD  & $ 46+ $\\
47AB  & $ 47+ $\\
48A  & $ 48||2+ $\\
50A  & $ 50+ $\\
51A  & $ 51+ $\\
52A  & $ 52||2+ $\\
52B  & $ 52||2+26 $\\
54A  & $ 54+ $\\
55A  & $ 55+ $\\
56A  & $ 56+ $\\
56BC  & $ 56||4+14 $\\
57A  & $ 57||3+ $\\
59AB  & $ 59+ $\\
60A  & $ 60||2+ $\\
60B  & $ 60+ $\\
60C  & $ 60+4,15,60 $\\
60D  & $ 60+12,15,20 $\\
60E  & $ 60||2+5,6,30 $\\
60F  & $ 60||6+10 $\\
62AB  & $ 62+ $\\
66A  & $ 66+ $\\
66B  & $ 66+6,11,66 $\\
68A  & $ 68||2+ $\\
69AB  & $ 69+ $\\
70A  & $ 70+ $\\
70B  & $ 70+10,14,35 $\\
71AB  & $ 71+ $\\
78A  & $ 78+ $\\
78BC  & $ 78+6,26,39 $\\
84A  & $ 84||2+ $\\
84B  & $ 84||2+6,14,21 $\\
84C  & $ 84||3+ $\\
87AB  & $ 87+ $\\
88AB  & $ 88||2+ $\\
92AB  & $ 92+ $\\
93AB  & $ 93||3+ $\\
94AB  & $ 94+ $\\
95AB  & $ 95+ $\\
104AB  & $ 104||4+ $\\
105A  & $ 105+ $\\
110A  & $ 110+ $\\
119AB  & $ 119+ $\\
\end{tabular}

\end{multicols}

\subsection*{Ordering by asymptotic magnitude of McKay-Thompson coefficients} The asymptotic formulas in Theorem \ref{core} give rise to the following natural order of the conjugacy classes corresponding to the asymptotic magnitude of their McKay-Thompson coefficients $c_g(-m, n)$. 
A class is understood to be removed if $c_g(-m, n)$ vanishes, which if it does, occurs periodically with $n$. We list classes X and Y in brackets $\{\mathrm{X, Y, \ldots}\}$ when
\[\lim_{n \rightarrow \infty} \frac{|c_g(-m, n)|}{|c_{g'}(-m, n)|} = 1\]
 for $g$ in X and $g'$ in Y. We group classes as $(\mathrm{X, Y,\ldots})$ if the above limit takes on another constant value, which may depend on congruence properties of $m$ and $n$. The order within round brackets may change depending on these congruence conditions. Finally, for classes with exceptions listed in Theorem \ref{core}, we write $[\mathrm{X}]$ to indicate the placement of a class outside of the exception and $[\mathrm{X}]_{(*)}$ to refer to the placement when $m$ and $n$ satisfy the hypotheses of exception $(*)$ in Theorem \ref{core}.

\begin{align*}
&\mathrm{1A, 2A, 3A, \{2B, 4A\}, 5A, \{6A, 6B\}, 7A, (4B, 8A), (3C, 3B, 9A), \{10A, 10D\}, 11A,}\\
&\mathrm{\{6C, 12A,12H\}, 13A, \{14A, 14C\}, \{15A, 15C\}, (\{4C, 4D,8B\}, 16C), 17A, (6D, \{18B, 18E\}),} \\
&\mathrm{19A, \{10B, 20A, 20F\}, \{21A, 21D\}, 22A, 23AB, (\{12C, 12F\}, \{24B, 24I\}), (5B, 25A ),} \\
&\mathrm{\{26A, 26B\}, \{27A,27B\}, \{14B, 28B \}, 29A, \{30B, 30D, 30F\}, 31AB, (8C,16A, 32A),} \\
&\mathrm{33B, 34A, \{35A, 35B\}, (\{6F,12D\}, 6E, 12B, \{36A, 36D\}), 38A, \{39A, 39CD\}, \{20B, 20E\}, } \\
&\mathrm{41A, \{42A, 42D\},\{22B, 44AB\}, (15B, 45A), 46CD, 47AB, \{12E, 24A, 24H\}, 7B, (10C, 50A),} \\
&\mathrm{51A, 54A, 55A, (\{28A, 28D\}, 56A), 59AB, \{30A, 30C, 30G,60B, 60C\}, 62AB, 21C,} \\
&\mathrm{(\{[8D], [8E], 8F\}, 32B), \{66A, 66B\}, 69AB, 70A, 71AB, (12G, 36C, 24C), 78A,} \\
&\mathrm{(20D, \{40B, 40CD\}), 9B, 42B, 87AB, 30E, \{46AB, 92AB\}, 94AB, 95AB, (24F, 48A),} \\
&\mathrm{ 33A, (10E, 20C), \{52A, 52B\}, 105A, 110A, 28C, 39B, 119AB, \{60A, 60E\}, 68A, 70B,} \\
&\mathrm{(\{12J,24E\}, 12I), 21B, 78BC, 40A, [18A], 84A, 13B, 57A, 88AB, 60D, [24D], 56BC, } \\
&\mathrm{ 15D, \{42C, 84C\}, ([16B], \{[8D]_{(1)}, [8E]_{(1)}\}), 93AB, 24G, \{[18A]_{(3)}, 18D, 36B\}, 84B, 60F,}\\
&\mathrm{104AB, (24J, [24D]_{(4)}), [16B]_{(2)}}
\end{align*}


\begin{thebibliography}{10}

\bibitem{me}
L.~Beneish and H.~Larson, \textit{Traces of singular values of Hauptmoduln}, Int.~J.~Numb.~Th. \textbf{11}, 1027 (2015).

\bibitem{B}
R.~Borcherds, \textit{Vertex algebras, Kac-Moody algebras, and the Monster}, Proc.~Nat.~Acad.~Sci.~U.S.A. \textbf{83} (1986), no.~10, 3068--3071.

\bibitem{BO}
K.~Bringmann and K.~Ono, \textit{Coefficients of harmonic mass forms}, Partitions, q-Series and Modular Forms, Developments in Mathematics, vol.~23, Springer New York, 2012, 23--38.

\bibitem{Brunier}
J.~H.~Brunier, \textit{Borcherds products on O(2,l) and Chern classes of Heegner divisors}, Springer Lect.~Notes Math.~\textbf{1780}, Springer-Verlag (2002).

\bibitem{57}
J.~H.~Conway, R.~T.~Curtis, S.~P.~Norton, R.~A.~Parker, and R.~A.~Wilson, \textit{Atlas of finite groups. Maximal subgroups and ordinary characters for simple groups. With compt.~assist.~from J.~G.~Thackray}, Oxford, Clarendon Press, 1985.

\bibitem{CN}
J.~H.~Conway and S.~P.~Norton, \textit{Monstrous moonshine}, Bull.~London Math.~Soc.~\textbf{11} (1979), no.~3, 308--339.

\bibitem{85}
J.~F.~R.~Duncan and I.~B.~Frenkel, \textit{Rademacher sums, moonshine and gravity}, Commun.~Number Theory Phys.~\textbf{5} (2011), no.~4, 1--128.

\bibitem{DGO}
\newblock J.~F.~R.~Duncan, M.~Griffin, and K.~Ono,
\newblock \textit{Moonshine}, Research in the Mathematical Sciences, \textbf{2} (2015), A11.

\bibitem{FLM}
I.~Frenkel, J.~Lepowsky, and A.~Murman, \textit{A natural representation of the Fischer-Griess Monster with the modular function $J$ as character}, Proc.~Nat.~Acad.~Sci.~U.S.A. \textbf{81} (1984), no.~10, Phys.~Sci., 3256--3260.

\bibitem{FLM2}
I.~Frenkel, J.~Lepowsky, and A.~Murman, \textit{A moonshine module for the moster},  Math.~Sci.~Res.~Inst.~Publ., vol.~3, Springer, New York, 1985, 231--273.

\bibitem{HL}
K.~Harada and M.~L.~Lang, \textit{The McKay-Thompson series associated with the irreducible characters of the Monster}, Moonshine, the Monster, and related topics (South Hadley, MA, 1994), Contemp.~Math., vol.~193, Amer.~Math.~Soc., Providence, RI, 1996, 93--111.

\bibitem{Kilford}
L.~J.~P. Kilford.
\newblock {\em Modular forms}.
\newblock Imperial College Press, London, 2008.
\newblock A classical and computational introduction.

\bibitem{T}
J.~G.~Thompson, \textit{Some numerology between the Fischer-Griess Monster and the elliptic modular function}, Bull.~London Math.~Soc.~\textbf{11} (1979), no.~3, 352--353.



\end{thebibliography}
\end{document}